\documentclass[11pt,reqno]{amsart}
\usepackage{amssymb,latexsym}
\usepackage{graphicx}
\usepackage{fancyhdr,amssymb}
\usepackage{url}		
\usepackage{ tipa }

\theoremstyle{plain}
\numberwithin{equation}{section}
\newtheorem{thm}{Theorem}[section]
\newtheorem{theorem}[thm]{Theorem}

\newtheorem{definition}[thm]{Definition}
\newtheorem{corollary}[thm]{Corollary}

\begin{document}
\fancyhead{}
\renewcommand{\headrulewidth}{0pt}
\fancyfoot{}
\fancyfoot[LE,RO]{\medskip \thepage}
\fancyfoot[LO]{\medskip MONTH YEAR}
\fancyfoot[RE]{\medskip VOLUME , NUMBER }

\setcounter{page}{1}

\title[Prime Fibonacci Sequences]{A Note on Prime Fibonacci Sequences}
\author{Jeremy F.~Alm and Taylor Herald}
\address{Department of Mathematics\\
                Illinois College\\
                Jacksonville, IL\\
                62650}
\email{alm.academic@gmail.com}
\email{herald.taylor@mail.ic.edu}

\begin{abstract}
In this paper, we define a variant of  Fibonacci-like sequences that we call prime Fibonacci sequences, where one takes the sum of the previous two terms and returns the smallest odd prime divisor of that sum as the next term.  We prove that these sequences always terminate in a power of two but can be extended infinitely to the left. 
\end{abstract}

\maketitle

\section{Introduction}

In \cite{GKS}, Guy, Khovanova, and Salazar study a variant of Fibonacci-like sequences that they call subprime Fibonacci sequences---a variant suggested by Conway.  To compute a term of a subprime Fibonacci sequence, one takes the sum of the previous two terms and, if the sum is composite, divides by its smallest prime divisor.  They study periodic subprime Fibonacci sequences and derive many interesting results; however, they are unable to prove that any such sequence ``diverges'' (i.e., does not eventually end in a cycle).  Indeed, it is difficult to imagine how one might prove such a thing.  The question, ``Do all subprime Fibonacci sequences eventually end in a cycle?"  may belong to the class of extremely difficult (possibly even formally unsolvable) problems discussed in \cite{Conway}; one such example is the generalized Collatz problem, which was shown in \cite{KS} to be undecidable.

In this paper, we consider a different variant of the Fibonacci sequence that submits itself to fairly complete analysis: instead of adding two terms and \emph{dividing} by the smallest prime divisor we \emph{return} the smallest \emph{odd} prime divisor as the next term.

\begin{definition}
Let $p_1,p_2$ be odd primes.  Then the prime Fibonacci sequence generated by $p_1,p_2$ is $(a_i)$, where $a_{i+2}$ is the smallest odd prime divisor of $a_{i}+a_{i+1}$ if $a_{i}+a_{i+1}$ is not a power of 2; otherwise, the sequence terminates.
\end{definition}

For example, consider starting with $5,7$:
\[
5,7,3,5. 
\]
The sequence terminates because the last two terms add up to a power of two.  As we will see, these prime Fibonacci sequences always terminate, but can be made arbitrarily long.  In fact, they can be infinite to the left, so we will take one and ``turn it around" to get an infinite sequence.

\section{Proofs of Main Results}

\begin{theorem} \label{terminates}
Given distinct odd primes $p_1,p_2$, the prime Fibonacci sequence $(a_i)$ with $a_1=p_1, a_2=p_2$ terminates in a power of 2.
\end{theorem}

\begin{proof}
Let $(a_i)$ be any prime Fibonacci sequence. Since $a_i$ and $a_{i+1}$ are odd, their sum is even. So, $a_{i+2} \leq \frac{a_i + a_{i+2}}{2}$. In particular, $a_{i+2} < \max[a_i, a_{i+1}]$ if $a_i \neq a_{i+1}$. (It is easy to show both that any eventually constant prime Fibonacci sequence must have been constant from the beginning and that no nontrivial periodic sequences exist.) The conclusion follows.
\end{proof}



\begin{theorem} \label{main}
Given distinct odd primes $p_1,p_2$, we can always find an odd prime $p_0$ so that $p_2$ is the smallest odd prime dividing $p_0+p_1$.
\end{theorem}

\begin{proof}
Let $p_1, p_2$ be distinct odd primes.  If $p_2\mid p_0+p_1$, then $p_2\cdot m=p_0+p_1$ for some $m\in\mathbb{Z}^+$, and $p_0=p_2\cdot m-p_1$.  Let 
\[
2,3,5,\dots,p',p_2
\]

be the list of primes up to $p_2$.  Let $Q=\prod_{2<p\leq p_2} p$, and consider the equivalences
\begin{align}
x &\equiv 1-p_1 \pmod{3} \label{eq:first}\\
x &\equiv 1-p_1 \pmod{5}\\
x &\equiv 1-p_1 \pmod{7}\\
&\vdots \nonumber \\
x &\equiv 1-p_1 \pmod{p'}\\
x &\equiv-p_1 \pmod{p_2} \label{eq:last}
\end{align}

By the Chinese Remainder Theorem, there is a unique solution to \eqref{eq:first}-\eqref{eq:last} modulo $Q$;  let $a$ be this solution.  Note that $(a,Q)=1$.  Then Dirichlet's Theorem on primes in arithmetic progressions (see \cite{Stein} for a proof) tells us that the sequence 
\[
a, a+Q, a+2Q, a+3Q, \dots
\]
contains infinitely  many primes.  Let $p_0$ be the smallest of these.  We verify that $p_0$ has the desired properties.

Write $p_0=a+nQ$ for some $n\in\mathbb{Z}$, so 
\begin{align*}
p_0+p_1 &=a+nQ+p_1\\
&\equiv -p_1+0+p_1 \pmod{p_2}\\
&=0.
\end{align*}
Hence $p_2\mid p_0+p_1$.  Now let $q<p_2$ be an odd prime.  Again, 
\begin{align*}
p_0+p_1 &=a+nQ+p_1\\
&\equiv 1-p_1+0+p_1 \pmod{q}\\
&\equiv 1 \pmod{q}.
\end{align*}
Hence $q\nmid p_0+p_1$.

\end{proof}

\begin{corollary} \label{cor:kterm}
For all positive integers $k$, there is a  prime Fibonacci sequence of length at least $k$.
\end{corollary}

\begin{corollary} \label{cor:left}
Any prime Fibonacci sequence can be extended indefinitely to the left.
\end{corollary}

We can give a  second proof of Corollary \ref{cor:kterm} using the celebrated Green-Tao theorem \cite{GT} that for all $n$ one can find $n$ primes in arithmetic progression.

\begin{proof}[Alternate proof of Corollary \ref{cor:kterm}]
We construct a prime Fibonacci sequence of length $k$.  By \cite{GT}, there is an arithmetic progression $p_0,p_1,\dots, p_n$ of primes of length $2^{k-2}+1$ (so $n=2^{k-2}$).  We let $a_1=p_0$ and $a_2=p_n$. Notice that, because the $p_i$'s  are in arithmetic progression, $p_0+p_n = 2p_{n/2}$, so $a_3=p_{n/2}$. Similarly, $a_2+a_3=p_n+p_{n/2} = 2p_{3n/4}$, so $a_4=p_{3n/4}$. In each case, once $a_i = p_\alpha$ and $a_{i+1} = p_\beta$ have been found, $a_{i+2} = p_{(\alpha+\beta)/2}$.  The claim that this works is equivalent to the claim that the sequence
\[
    b_1=0, b_2 =2^{k-2}, \textrm{ and } b_{i+2} = \frac{b_i + b_{i+1}}{2}
\]
is an integer for $i\leq k$, which is easily verified by induction.

\end{proof}

\section{Reversed Prime Fibonacci Sequences}

In light of Corollary \ref{cor:left}, we make the following definition:

\begin{definition}
Let $p,q$ be distinct odd primes.  Then $(a_i)$ is the reversed prime Fibonacci sequence generated by $p,q$ if $a_1=p,a_2=q$, and for all $i\geq 1$, $a_{i+2}$ is the smallest odd prime with the following property: $a_i$ is the smallest odd prime divisor of $a_{i+1}+a_{i+2}$.
\end{definition}
Then we have the following result:

\begin{theorem}
Let $(a_i)$ be an eventually monotonic reversed prime Fibonacci sequence.  Then $(a_i)$ has asymptotic density zero in the primes.
\end{theorem}

\begin{proof}
Let $(a_i)$ be any (eventually) monotonic reversed prime Fibonacci sequence, and let $(p_i)$ be the sequence of primes.  Since $a_i\mid (a_{i+1}+a_{i+2})$ and $(a_i)$ is monotonic, we have $2a_i<a_{i+1}+a_{i+2}$.  Since $a_i$ is a divisor of $a_{i+1}+a_{i+2}$, we get the stronger inequality $4a_i\leq a_{i+1}+a_{i+2}$.  Then we have $a_{i+2}\geq 4a_j-a_{i+1}$.  Consider a sequence $(b_i)$ satisfying the recurrence $b_{i+2}=4b_i-b_{i+1}$.  This sequence has characteristic equation $r^2+r-4=0$, with positive root $\alpha=\frac{1}{2}(\sqrt{17}-1)\approx 1.56$.  Certainly, we have $a_n\gg b_n$, no matter the initial conditions on either sequence, so $a_n\gg\alpha^n$.  Note also that from \cite{Jaroma} we have $p_n\ll (1.2)^n$.  Combining these, we get 
\[
\frac{p_n}{a_n}\ll\frac{1.2^n}{\alpha^n}<\frac{1.2^n}{1.5^n}\rightarrow 0.
\]
\end{proof}
We do not know whether any reversed prime Fibonacci sequence is eventually monotonic. Getting good computational data is difficult, because although reversed prime Fibonacci sequences may not be monotonic, they still grow quite quickly.  For instance, consider the sequence below beginning with $3,5$, for which we can compute just 15 terms:

\[
3, 5, 7, 3, 11, 7, 37, 19, 277, 331, 223, 439, 7, 406507, 67, \dots
\]
(This is sequence number A255562 in OEIS \cite{sloane}.)  We do not know the next term, but we do know by exhaustive search that it must be greater than two billion.  



\end{document}